\documentclass{article}
\usepackage[a4paper]{geometry}
\usepackage{amsmath}
\usepackage{amssymb}
\usepackage{amsthm}
\usepackage{mathrsfs}
\usepackage{mathtools}

\usepackage{tikz}
\usepackage{cleveref}
\usepackage{tikz-cd}
\usepackage{mathtools}

\usepackage{placeins}

\usepackage{float}

\usepackage{comment}

\usepackage{enumerate}

\theoremstyle{definition}
\newtheorem{thm}{Theorem}[section]
\crefname{thm}{Theorem}{Theorems}

\crefname{prop}{Proposition}{Propositions}
\newtheorem{lem}[thm]{Lemma}
\crefname{lem}{Lemma}{Lemmas}

\newtheorem{conj}[thm]{Conjecture}
\newtheorem{quest}[thm]{Question}

\crefname{defn}{Definition}{Definitions}

\newtheorem*{ack*}{Acknowledgements}

\newcommand{\co}{\operatorname{co}}

\newcommand{\gap}{\operatorname{gap}}

\newcommand{\mb}[1]{\mathbb{#1}}
\newcommand{\sub}{\subseteq}

\newcommand{\tT}{\theta}

\newcommand{\Ss}{\Sigma}
\newcommand{\sm}{\setminus}
\newcommand{\es}{\emptyset}

\usepackage{titling}

\title{On Ruzsa’s discrete Brunn-Minkowski conjecture}
\author{Peter van Hintum\thanks{New College, University of Oxford, UK. email: peter.vanhintum@new.ox.ac.uk}, Peter Keevash\thanks{Mathematical Institute, University of Oxford, UK. Supported by ERC Advanced Grant 883810.}, Marius Tiba\thanks{Mathematical Institute, University of Oxford, UK. email: marius.tiba@maths.ox.ac.uk}}
\allowdisplaybreaks

\begin{document}
\maketitle
\begin{abstract}
We prove a conjecture by Ruzsa from 2006 on a discrete version of the Brunn-Minkowski inequality,
stating that for any $A,B\subset\mathbb{Z}^k$ and $\epsilon>0$
with $B$ not contained in $n_{k,\epsilon}$ parallel hyperplanes
we have $|A+B|^{1/k}\geq |A|^{1/k}+\left(1-\epsilon\right)|B|^{1/k}$.
\end{abstract}

\section{Introduction}

The Brunn-Minkowski inequality is a foundational result for Convex Geometry and Asymptotic Geometric Analysis,
with diverse connections to other areas including Information Theory, Statistical Mechanics and Algebraic Geometry
(see the survey by Gardner \cite{gardner2002brunn}). 
In the continuous setting, for measurable $A,B\subset\mathbb{R}^k$ it states that
$|A+B|^{1/k} \ge |A|^{1/k} + |B|^{1/k}$. It is natural, and well-motivated 
by questions in Additive Combinatorics (see \cite{tao2006additive}),
to ask for an analogous inequality in the discrete setting of $A,B\subset\mathbb{Z}^k$. However, this task is surprisingly
difficult to define, let alone accomplish, as discussed in Ruzsa's ICM survey \cite{ruzsa2006additive}. 

Naturally, one should impose some non-degeneracy condition, otherwise the question reduces to subsets of $\mathbb{Z}$,
for which it is well-known and easy to show that $|A+B| \ge |A|+|B|-1$, with equality for (dilated) intervals.
Gardner and Gronchi  \cite{gardner2001brunn} 
determined the minimum value of $|A+B|$ in terms of $|A|$ and $|B|$ assuming that one or both
of $A$ and $B$ are full-dimensional. While this may at first seem to be a natural discrete analogue, it turns out that under
this weak non-degeneracy assumption the extremal examples are somewhat close to intervals, and the resulting bound 
on $|A+B|$ is much smaller than the `Brunn-Minkowski bound' $ |A|^{1/k} + |B|^{1/k}$. 
Green and Tao \cite{green2006compressions} showed that 
one can obtain an asymptotic Brunn-Minkowski bound assuming that $A$ and $B$ are dense subsets of a reasonable box.
Thus they obtained the correct bound, but under a very strong assumption; this motivates the question of proving this bound
under the natural minimal assumption, formulated by Ruzsa \cite[Conjecture 4.12]{ruzsa2006additive}. 

\begin{conj}\label{ruzsaconj}
For all $k \in \mathbb{N}$ and $\epsilon$ there exists $n_{k,\epsilon}$ so that if $A,B\subset\mathbb{Z}^k$ with
$B$ not covered by $n_{k,\epsilon}$ parallel hyperplanes then $|A+B|^{1/k}\geq |A|^{1/k}+(1-\epsilon)|B|^{1/k}$.
\end{conj}

Our main result establishes Conjecture \ref{ruzsaconj} in the following sharp form
that establishes the optimal quantitative dependence between the parameters.

\begin{thm}\label{quantcor}
For every $k\in \mathbb{N}$ and $t>0$ there exists $c^{\ref{quantcor}}_{k,t}$ so that
if $A,B\subset\mathbb{Z}^k$ with $t|B|\leq |A|$ and $B$ is not covered by $n$ parallel hyperplanes
then $|A+B|^{1/k}\geq |A|^{1/k}+\left(1-c^{\ref{quantcor}}_{k,t}n^{-1}\right)|B|^{1/k}$.
\end{thm}

To see that the dependence on $t$ is necessary, consider the example 
where $A$ is an interval in $\mathbb{Z} \subset \mathbb{Z}^k$
and $B$ is an interval together with $k+n$ points in general position,
with $t|B|=|A|=m \gg t^{-1} \gg n \gg k$. Here we have $|A+B| \sim |B| + n|A| = (1+nt)|B|$
and  $(|A|^{1/k}+(1-c/n)|B|^{1/k})^k \sim ((1-c/n)^k+kt^{1/k})|B|$,
so the conclusion of Theorem \ref{quantcor} cannot hold if $c < t^{1/k}$.

\subsection{Brunn-Minkowski and the additive hull}

We will deduce Theorem \ref{quantcor} from more general results
in which we bound $|A+B|$ for $A,B \subset \mathbb{Z}$,
where $B$ satisfies a $k$-dimensional non-degeneracy condition,
in the sense of the additive hull introduced in \cite{van2023locality}.
To describe this, we make the following definitions.

Given an axis-aligned box $C \subset \mb{R}^k$ and a linear map $\phi: \mb{Z}^k \to \mb{Z}$ we call 
$P=\phi(C \cap \mb{Z}^k)$ a $k$-dimensional generalised arithmetic progression ($k$-GAP) .
We say $P$ is $t$-proper if $\phi$ is injective on $tC \cap \mb{Z}^k$.
For $B \sub \mathbb{Z}$ we let $\gap_{n}^{k}(B)$ be a smallest set containing $B$ 
of the form $X+P$ where $|X|\leq n$ and $P$ is a ($1$-)proper $k$-GAP.
One can think of $\gap_{n}^{k}(B)$ as a parameterised family
of hulls that aim to capture the additive structure of $B$.

Theorem \ref{quantcor} will follow immediately from the following two results
in which we consider separately two regimes for $|B|/|A|$, as when $|B|/|A|$ is small 
we obtain a stronger bound not depending on this ratio.

\begin{thm}\label{main1}
For every $k\in \mathbb{N}$ and $t>0$ there exists $c^{\ref{main1}}_{k,t}, e^{\ref{main1}}_{k,t}$ so that 
if $A,B\subset\mathbb{Z}$ with $t|B|\leq |A|\leq t^{-1}|B|$
and $|\gap_{n}^{k-1}(B)|\geq e^{\ref{main1}}_{k,t}|B|$
then $|A+B|^{1/k}\geq |A|^{1/k}+(1-c^{\ref{main1}}_{k,t}n^{-1})|B|^{1/k}$.
\end{thm}

\begin{thm}\label{main2}
For every $k\in\mathbb{N}$ there exists $L^{\ref{main2}}_k$, $e^{\ref{main2}}_{k}$, $c^{\ref{main2}}_{k}$
so that the following holds for any $A,B\subset \mathbb{Z}$ 
where $L^{\ref{main2}}_k \le 2^j \le (|A|/|B|)^{1/k} \le 2^{j+1}$ with $j \in \mb{N}$.
If $|\gap_{\ell' n}^{k-1}(\ell' \cdot B)|\geq e^{\ref{main2}}_{k}\ell'^k|B|$
for any $\ell'=2^{i'}$, $i' \in\{0,\dots,j+1\}$ then 
$|A+B|^{1/k}\geq |A|^{1/k}+(1-c^{\ref{main2}}_{k}n^{-1})|B|^{1/k}$.
\end{thm}

The non-degeneracy conditions in Theorems \ref{main1} and \ref{main2}
might at first seem awkward to verify, but in fact they hold trivially when applying 
these theorems to prove Theorem \ref{quantcor}. Indeed, consider $A,B \sub \mb{Z}^k$
such that $B$ is not contained in $n$ parallel hyperplanes. We note for any $\ell' \in \mb{N}$
that the $\ell'$-fold sumset $\ell' \cdot B$ is not contained in $\ell' n$ parallel hyperplanes.
Now fix a linear map $\phi: \mb{Z}^k \to \mb{Z}$ with sufficiently disparate coefficients that is injective on a large box 
containing $A$, $A+B$ and $L^{\ref{main2}}_k \cdot B$.
We note that $|\gap_{\ell' n}^{k-1}(\ell' \cdot B)|>e^{\ref{main2}}_{k}\ell'^k|B|$ for all $\ell' \le L^{\ref{main2}}_k $.
Thus we can apply  Theorem \ref{main1} or \ref{main2} to $\phi(A)$ and $\phi(B)$
and deduce Theorem \ref{quantcor}. 

\subsection{Stability}

We will prove Theorems \ref{main1} and \ref{main2} by the Stability Method.
This is a well-known technique in Extremal Combinatorics, but apparently not as widely exploited 
in Additive Combinatorics prior to \cite{van2023locality}. 

The idea will be to describe the approximate structure of sets $A,B$ 
as in our theorems for which $|A+B|$ is approximately minimised, from which 
it will then be straightforward to deduce the required bound on $|A+B|$ 
from a variant of the Green-Tao discrete Brunn-Minkowski result mentioned above. 

When $A$ and $B$ are of roughly comparable sizes,
the following result shows that they are dense subsets 
of translates of the same box. 

\begin{thm}\label{stab1}
For every $k \in \mathbb{N}$ and $t>0$ there exist $ n^{\ref{stab1}}_{k,t}, e^{\ref{stab1}}_{k,t}, \Delta^{\ref{stab1}}_{k,t}, f^{\ref{stab1}}_{k,t}>0$ so that the following holds. 
Let $A,B\subset\mathbb{Z}$ with $t|B|\leq |A|\leq t^{-1}|B|$ such that
$|\gap_{n^{\ref{stab1}}_{k,t}}^{k-1}(B)|\geq e^{\ref{stab1}}_{k,t}|B|$
and  $|A+B|^{1/k}\leq |A|^{1/k} + (1+\Delta^{\ref{stab1}}_{k,t})|B|^{1/k}$. 
Then there is a $10$-proper $k$-GAP $P$ 
with  $|P| \leq f^{\ref{stab1}}_{k,t}|A|$
such that $A$ and $B$ are contained in translates of $P$.
\end{thm}

In the unbalanced case $|A| \gg |B|$ we obtain a similar structural result for $B$,
but for $A$ our description is much weaker (due to applying Pl\"unnecke's inequality),
so we will not explicitly state the stability result that follows from the proof. For the case $A=B$, a sharp stability result in this direction was established in $\mathbb{Z}^k$ in \cite{van2023sets}, which was extended to $\mathbb{Z}$ in \cite{van2023locality}.

\section{Lemmas and tools}

We start by gathering various tools and lemmas needed for the proofs of our theorems.

\subsection{Coverings}

Our main tool will be the following version of Freiman's Theorem with separation properties,
which is a special case of Theorem 2.2 and Theorem 1.4  from \cite{van2023locality}.
We require the following definitions for the statement.
Given an abelian group $G$ and $A,B \sub G$ with $0 \in B=-B$,
we say $A$ is $B$-separated if $a-a' \notin B$ for all distinct $a,a' \in A$.
We say that a $k$-GAP $P=\phi(C \cap \mb{Z}^k)$ is $n$-full 
if each side of $C$ has length at least $n$.

\begin{thm}[\cite{van2023locality}] \label{Freiman}
Let $d,s,m\in\mathbb{N}$, $n_1,\dots, n_d\in\mathbb{N}$, and $S\subset\mathbb{Z}$. 
If $|S+S|\leq 1.9\cdot 2^{d}|S|$ then for some $d' \le d$ we have $S\subset X+P$
where $P$ is an $s$-proper $n_{d'}$-full $d'$-GAP with $P=-P$
and $X+X$ is $m^2 \cdot P/m$-separated,
with $|X|=O_{d,n_{d'+1},\dots,n_{d}}(1)$ and $|P|=O_{d,s,m,n_{d'+1},\dots,n_{d}}(|S|)$.
Moreover, if $|S+S|\leq (2^{d'}+1/2)|S|$ then $|X|=1$.
\end{thm}

We also use the well-known Ruzsa Covering Lemma \cite{ruzsa1999analog}.

\begin{lem} \label{ruzsa}
Suppose $A,B\subset \mathbb{Z}$ with $|A+B|\leq K|B|$.
Then $A\subset X+B-B$ for some $X\subset \mathbb{Z}$ with $|X| \leq K$.
\end{lem}

\subsection{Iterated sumsets}

Here we collect various results on iterated sumsets.
We start with a version of Pl\"unnecke's inequality,
in a strengthened form due to Petridis \cite{petridis2012new}.

\begin{lem} \label{lem:P}
Let $A,B \sub \mb{Z}$ with
 $K = |A+B|/|A| = \min \{ |A'+B|/|A'|: \es \ne A'\subset A \}$.
Then $|A+\ell\cdot B|\leq K^\ell |A|$ for all  $\ell \in \mb{N}$.
\end{lem}

Next we require the following result, 
which is a special case of \cite[Theorem 1.4]{van2023locality},
establishing the case $A=B$ of Theorem \ref{quantcor}.
A less quantitative version of this result  
was established by Freiman (also presented by Bilu [Bil99]).
We remark (although it is not needed here) that the asymptotically 
optimal constant $c^{\ref{freimanbilu}}_{k}$ is determined in  \cite{van2023locality}.

\begin{thm}\label{freimanbilu}
For all $k\in\mathbb{N}$ there exists $c^{\ref{freimanbilu}}_{k},e^{\ref{freimanbilu}}_{k}$ 
so that if $A\subset\mathbb{Z}$ with $|\gap_{n}^{k-1}(A)|\geq e^{\ref{freimanbilu}}_{k}|A|$ 
then $|A+A|\geq 2^k(1-c^{\ref{freimanbilu}}_{k}n^{-1})|A|$.
\end{thm}

By iteration we can deduce the following estimate for iterated sumsets.

\begin{lem}\label{iteratedBilu}
For every $k$ there exist $c^{\ref{iteratedBilu}}_{k}, e^{\ref{iteratedBilu}}_{k}>0$ 
so that if $B\subset \mathbb{Z}$ with
$|\gap_{\ell' n}^{k-1}(\ell' \cdot B)|\geq e^{\ref{iteratedBilu}}_{k} \ell'^k|B|$
for all $\ell' = 2^{i'}$ with $0 \le i' \le j$ then 
 $|2^j \cdot B|\geq (1-c^{\ref{iteratedBilu}}_{k}n^{-1})2^{jk}|B|$.
\end{lem}

\begin{proof}[Proof of \Cref{iteratedBilu}]
We let $e^{\ref{iteratedBilu}}_{k} = 2^k e^{\ref{freimanbilu}}_k$ 
and $c^{\ref{iteratedBilu}}_{k} = 2 c^{\ref{freimanbilu}}_k$.
The case $j=0$ is trivial and the case $j=1$ follows from \Cref{freimanbilu}.
In general, for any $i<j$ we note that if $|2^i \cdot B| \leq 2^{(i+1)k}|B|$
then $|\gap^{k-1}_{2^in}(2^i\cdot B)| \geq  e^{\ref{freimanbilu}}_{k}|2^i \cdot B|$,
so we can apply \Cref{freimanbilu} to get
$$|2^{i+1}\cdot B|=|2^i\cdot B+2^i\cdot B|\geq (1-c^{\ref{freimanbilu}}_k(2^{i}n)^{-1})2^k|2^i\cdot B|. $$
If this holds for all $i<j$ then the lemma follows from
$$|2^j\cdot B|\geq |B|\prod_{i<j} 2^k(1-c^{\ref{freimanbilu}}_k( 2^{i}n)^{-1})
\geq 2^{jk}|B| \Big(1- c^{\ref{freimanbilu}}_kn^{-1}\sum_{i<j} 2^{-i}\Big)
\geq (1-2c^{\ref{freimanbilu}}_kn^{-1}) 2^{jk}|B|.$$
Otherwise, we can consider the largest $i_0 \le j$ with $|2^{i_0} \cdot B| \geq 2^{(i_0+1)k}|B|$,
note that $|2^{i_0+1} \cdot B| \geq |2^{i_0} \cdot B|$,
and similarly deduce the lemma from 
$|2^j\cdot B|\geq |2^{i_0+1} \cdot B|\prod_{i=i_0+1}^{j-1} 2^k(1-c^{\ref{freimanbilu}}_k( 2^{i}n)^{-1})$.
\end{proof}

Next we require the following result of Lev \cite{lev1996structure}
(see Corollary 1 and its proof) that gives a sharp bound for iterated sumsets
(although we state a simpler bound that suffices here).

\begin{thm}[\cite{lev1996structure}] \label{thm:lev}
Suppose $A \sub \{0,\dots,\ell\}$ with $0,\ell \in A$ and $\gcd(A)=1$.
Let $A' = \{0,\dots,|A|-2\} \cup \{\ell\}$ and $h \in \mb{N}$.
Then $|h \cdot A| \ge |h \cdot A'| \ge h\ell - \ell^2/(|A|-2)$.
\end{thm}
 
We deduce the following lemma that provides long arithmetic progressions in iterated sumsets.

\begin{lem} \label{lem:ap}
Suppose $A \sub [-\ell,\ell]$ with $0,\ell \in A=-A$ 
and $|A|> \tfrac{2\ell+1}{m+1}$, where $m \in [\ell/2]$.
Then $r[-2m\ell',2m\ell'] \sub 20m \cdot A$, where $r := \gcd(A) \in [m]$ and $\ell'=\ell/r$.
\end{lem}

\begin{proof}
We first note that $\tfrac{2\ell+1}{m+1} < |A| \le (2\ell+1)/r$, so $r \le m$. 
We apply \Cref{thm:lev} to $A' := A/r \cap [0,\ell']$,
obtaining $|10m \cdot A'| \ge 10m\ell' - \ell'^2/(\ell/m-2)  \ge 8m\ell'$.
Thus $10m \cdot A'$ has density at least $4/5$ in $\co(10m \cdot A') = [0,10m\ell']$.
Now fix any $z \in [-2m\ell',2m\ell']$ and consider
$S := \{ (x,y) \in [-10m\ell',10m\ell']^2: x<y, x+y=z \}$. 
Then $|S| \ge 16m\ell'$ and $\{x,y\} \sub 10m \cdot (A' \cup -A')$
for all but at most $4m\ell'$ pairs $(x,y) \in S$.
Thus we can write $z=x+y$ with $\{x,y\} \sub 10m \cdot (A' \cup -A')$,
and so $rz \in 20m \cdot A$.
\end{proof}

Now we apply the previous lemma to the following one,
that takes a covering of an iterated sumset $\ell \cdot B$ by boxes
and shrinks these boxes to obtain a covering by $B$;
crucially, the size of the boxes scales correctly with $\ell$,
up to a constant only depending on the number of boxes.

\begin{lem} \label{shrinkingthebox}
Suppose $P$ is a $40m\ell$-proper $k$-GAP
and $\ell\cdot B\sub X+\ell \cdot P$ with $|X| \le m$.
Then $X+ 20m m! \cdot P/m!$ contains a translate of $B$.
\end{lem}

\begin{proof}
By translating, we can assume $0 \in B$.
By enlarging $P$, we can assume $P=-P$ and $P$ is $20m\ell$-proper.
We can also assume $\ell > 20m$, otherwise we are done
by $B \sub \ell \cdot B \sub X+\ell \cdot P \sub X + 20mm!\cdot P/m!$.

Fix any $b \in B$ and consider $\{0,b,\dots,\ell b\} \sub \ell \cdot B \sub X+\ell \cdot P$.
As $|X| \le m$, we can find $Z \sub [0,\ell]$ with $|Z| \ge (\ell+1)/m \ge 4$ 
such that $bZ$ is contained in a single translate of $\ell \cdot P$.
We note that $Z-Z \sub [-\ell,\ell]$ is symmetric and 
$|Z-Z| \ge 2 \lceil (\ell+1)/m \rceil - 1 > (2\ell+1)/(m+1)$. 

By \Cref{lem:ap} we find $r[-2m\ell',2m\ell'] \sub 20m \cdot (Z-Z)$,
where  $r := \gcd(Z-Z) \in [m]$ and $\ell'=r^{-1}\max (Z-Z) \ge |Z|-1 \ge \ell/2m$.
As $b(Z-Z) \sub 2\ell \cdot P$ we deduce
$r[-\ell,\ell]b \sub 20m[-\ell,\ell]b \sub 20m\ell \cdot P$.

To complete the proof, it suffices to show that $b \in 20m \cdot P/r$.
To see this, recalling that  $P$ is $20m\ell$-proper,
we write $P=\phi(C \cap \mb{Z}^k)$ for some  box $C \subset \mb{R}^k$ and
linear map $\phi: \mb{Z}^k \to \mb{Z}$ injective on  $m!\ell C \cap \mb{Z}^k$.
Then $\phi^{-1}(r[-\ell,\ell]b)$ is an arithmetic progression in $20m\ell C$,
so $\phi^{-1}(rb) \in 20m C$, giving $b \in 20m \cdot P/r$.
\end{proof}

\subsection{Brunn-Minkowski in boxes}

In the following lemma we adapt a Brunn-Minkowski result of Green and Tao \cite{green2006compressions}
to the setting of dense sets in unbalanced boxes. For $t \in \mb{R}$ we write $(t)_+ = \max\{t,0\}$.

\begin{lem}\label{doublinginGAPs}
Let $P$ be an $n$-full $(\ell+1)$-proper $d$-GAP and let $Y\subset P$, $Z\subset \ell\cdot P$ be non-empty.
Then $$|Y+Z|^{1/d} \geq (|Y|-dn^{-1}|P|)_+^{1/d}+|Z|^{1/d}.$$
\end{lem}

\begin{proof}
Following Green and Tao \cite{green2006compressions},
we introduce a cube summand:
for non-empty sets $A,B\subset \mathbb{Z}^d$ we have 
(writing the measure space in the index)
\begin{align} \label{eq1}
\left|A+B+\{0,1\}^d\right|^{1/d}_{\mathbb{Z}^d}&=\left|A+B+\{0,1\}^d+[0,1]^d\right|^{1/d}_{\mathbb{R}^d}=\left|(A+[0,1]^d)+(B+[0,1]^d)\right|^{1/d}_{\mathbb{R}^d} \nonumber \\
&\geq \left|A+[0,1]^d\right|_{\mathbb{R}^d}^{1/d}+\left|B+[0,1]^d\right|^{1/d}_{\mathbb{R}^d}=\left|A\right|_{\mathbb{Z}^d}^{1/d}+\left|B\right|^{1/d}_{\mathbb{Z}^d}.
\end{align}
We may assume $P=\prod_{i=1}^d [0,n_i]$, where each $n_i\geq n$. 
Compress $Y$ and $Z$ onto the coordinate planes and note that 
this does not increase $|Y+Z|$ (see \cite[Lemma 2.8]{green2006compressions}). 
Letting $Y':=\{y\in Y: y+\{0,1\}^d\subset Y\}$, we note that if $x,x+(1,1,\dots, 1)\in Y$
then $x+\{0,1\}^d\subset Y$. As $Y$ is compressed, this implies that $Y\setminus Y'$ 
contains at most one point in every translate of $\mathbb{R}(1,\dots,1)$. 
There are fewer than $\sum_{i}|P|/n_i\leq dn^{-1}|P|$ such lines intersecting $P$, 
so $|Y'|\geq |Y|-dn^{-1}|P|$. We may assume $ |Y|-dn^{-1}|P| > 0$, otherwise
the lemma is immediate from $|Y+Z|\geq |Z|$. As $Y' + \{0,1\}^d\subset Y$,
applying \eqref{eq1} to $Y'$ and $Z$ proves the lemma.
\end{proof}

The inequality in the following lemma will arise when applying
the preceding lemmas to sets covered by several boxes.
For the statement we let $\Ss$ denote counting measure on $\mb{Z}$,
so that $\Ss(f) = \sum_{x \in \mb{Z}} f(x)$.

\begin{lem}\label{kconcaveprekopa}
Let $f,g,h\colon\mathbb{Z}\to \mathbb{R}_{\geq 0}$ so that $\Ss(f), \Ss(g)>0$ 
and for all $x,y\in\mathbb{Z}$ with $f(x),g(y)>0$ we have $h(x+y)^{1/d}\geq f(x)^{1/d}+g(y)^{1/d}$. 
Then $\Ss(h)^{1/d}\geq \Ss(f)^{1/d} + \Ss(g)^{1/d}$.

Moreover, if $f$ or $g$ is supported on more than one integer then
$\Ss(h)^{1/d}\geq (1+\Delta^{\ref{kconcaveprekopa}}_{t,d})( \Ss(f)^{1/d} + \Ss(g)^{1/d} )$,
where $\Delta^{\ref{kconcaveprekopa}}_{t,d}>0$ only depends on $d$ and $t:=\Ss(f)/\Ss(g)$. 
\end{lem}

\begin{proof}
The first statement follows from the Brunn-Minkowski inequality
as applied to $A_f + A_g \sub A_h$ where
$A_f := \bigcup_{x\in\mb{Z}} xe_1+\epsilon\left(0,f(x)^{1/d}\right)^d$,
noting that $|A_f| = \epsilon^d \Ss(f)$, etc.
The second statement follows by stability (see e.g.~\cite{van2023sharp} or \cite{figalli2017quantitative}),
noting that if $f$ is supported on more than one integer
then $|\co(A_f)|/|A_f| \to \infty$ as $\epsilon \to 0$. 
\end{proof}

\section{Proofs}

We are now ready to prove our main results.
Considering \Cref{doublinginGAPs} (Brunn-Minkowski in boxes),
we see that  \Cref{stab1} implies \Cref{main1}.
Thus it remains to prove  \Cref{stab1} and \Cref{main2}.

\begin{proof}[Proof of \Cref{stab1}]
Let $k,d \in \mb{N}$, $t \in (0,1)$ and fix constants satisfying
\[ k/t \ll d \ll n_d \ll \dots \ll n_1 \ll r \ll n. \]

Let $A,B\subset\mathbb{Z}$ with $t|B|\leq |A|\leq t^{-1}|B|$ 
such that $|\gap_n^{k-1}(B)|\geq n|B|$
and  $|A+B|^{1/k}\leq |A|^{1/k} + (1+n^{-1})|B|^{1/k}$.
Write $D = |\co(A \cup B)|$. After translation, we can assume that 
the distance between $A$ and $B$ is at least $rD$.
Let $S = A \cup B$. By Pl\"unnecke's inequality, we find 
$|A+A|,|B+B|=O_{k,t}(|A|)$, so $|S+S|=O_{k,t}(|S|)$.
We fix $d=O_{k,t}(1)$ so that $|S+S|/|S|\leq 1.9 \cdot 2^d$.
Then we apply  \Cref{Freiman} to $S$ with $s=m=10$ 
and $n_i$ as above to obtain $S\subset X+P$, where $P$ is 
a $10$-proper $n_{d'}$-full $d'$-GAP with $P=-P$ for some  $d' \le d$ 
and $X+X$ is $2 \cdot P$-separated, such that $|X| + |P|/|S| \ll n_{d'}$.
By the non-degeneracy condition $|\gap_n^{k-1}(B)|\geq n|B|$
we see that $d' \ge k$.

Next we claim that any translate of $P$ intersects only one of $A$ and $B$.
To see this, we write $P=\phi(C \cap \mb{Z}^{d'})$ for some box 
$C = \prod_{i=1}^{d'} [0,a_i]$, where each $a_i \ge n_{d'}$.
Due to the distance between $A$ and $B$, 
we can find $i \in [d']$ such that $x_i := \phi(b_i e_i)$
with $b_i := \lfloor a_i/n_{d'} \rfloor$ satisfies $|x_i| > 2D$.
Writing $P = [0,a_i]\phi(e_i) + P'$, we see that at most $b_i$ translates
$t\phi(e_i) + P'$ intersect $A$ or $B$. However, this gives 
$|A|+|B| \le |X| |P|b_i/a_i$, which contradicts $|X| + |P|/|S| \ll n_{d'}$.
Thus the claim holds, so we can find disjoint  $Y,Z\subset X$ 
with  $A\subset Y+P$ and $B\subset Z+P$.

For each $y \in Y$ and $z \in Z$ we write
$A_y:=A\cap y+P$, $B_z:= B\cap z+P$
and $(A+B)_{y+z} := (A+B) \cap (y+z+P+P)$.
As $X+X$ is $2 \cdot P$-separated,
the sets $A_y+B_z$ are pairwise disjoint.
By \Cref{doublinginGAPs}, we have
$|A_y+B_z|^{1/d'}\geq |A_y|^{1/d'}+(|B_z|-d'n_{d'}^{-1}|P|)^{1/d'}$.
Thus we can apply \Cref{kconcaveprekopa} to the functions
$f(y) = |A_y|$, $g(z) = |B_z|-d'n_{d'}^{-1}|P|)$, $h(x)=|(A+B)_x|$
(extended to the domain $\mb{Z}$ with zeroes where undefined), obtaining
\[  |A+B|^{1/d'} \ge \Ss(h)^{1/d'} \ge \Ss(f)^{1/d'} + \Ss(g)^{1/d'}
\ge Q := |A|^{1/d'} + \big(  (1-n_{d'}^{-0.9}) |B| \big)^{1/d'}, \]
where we used  $|X| + |P|/|S| \ll n_{d'}$.
to estimate $\Ss(g) \ge |B|-|X|d'n_{d'}^{-1}|P| \ge (1-n_{d'}^{-0.9}) |B|$.

We claim that this implies $d'=k$. Indeed, suppose $d'>k$, let $\tT = Q^{-d'} |A|$,
and note that $\min\{\tT,1-\tT\} > t/2$ as $t|B|\leq |A|\leq t^{-1}|B|$.
Let $S := (1-t/2)^{1/k - 1/d'}$ and consider
$\tT^{1/k} + (1-\tT)^{1/k} \le S ( \tT^{1/d'} +  (1-t)^{1/d'} ) =  S$.
We deduce
\[ |A+B|^{1/k} \ge Q^{d'/k} \ge S^{-1} Q^{d'/k} (\tT^{1/k} + (1-\tT)^{1/k})
=  S^{-1} ( |A|^{1/k} + \big(  (1-n_{d'}^{-0.9}) |B| \big)^{1/k} ). \]
However, this contradicts  $|A+B|^{1/k}\leq |A|^{1/k} + (1+n^{-1})|B|^{1/k}$, so we deduce $d'=k$. 

Moreover, considering the stability statement in \Cref{kconcaveprekopa},
we deduce that $|Y|=|Z|=1$, i.e.\ $A$ and $B$ are each contained in one translate of $P$.
\end{proof}

\begin{proof}[Proof of \Cref{main2}]
Fix constants satisfying $n \ge c \ge L \gg k$.
Let  $A,B\subset \mathbb{Z}$ where $L \le \ell \le (|A|/|B|)^{1/k} \le 2\ell$ 
with $\ell = 2^j$, $j \in \mb{N}$.
Suppose $|\gap_{\ell' n}^{k-1}(\ell' \cdot B)|\geq n\ell'^k|B|$
for any $\ell'=2^{i'}$, $i' \in\{0,\dots,j+1\}$.
 
Fix $\es \ne A'\subset A$ which minimizes $|A'+B|/|A'|$.
It suffices to show the result for $A'$,
i.e.\ that $|A'+B|^{1/k}\geq |A'|^{1/k}+(1-cn^{-1})|B|^{1/k}$,
as this will imply $(|A+B|/|A|)^{1/k} \ge (|A'+B|/|A'|)^{1/k}
\ge 1 + (1-cn^{-1})(|B|/|A|)^{1/k}$.

We can assume $|A'+B|/|A'| \le |A+B|/|A| \le (1 + (|B|/|A|)^{1/k})^k$.
By \Cref{lem:P} (Petridis' form of Pl\"unnecke's inequality) we deduce
\[ |\ell\cdot B| \le  |A'+\ell\cdot B| \leq |A'|(1+(|B|/|A|)^{1/k})^{k\ell} \le |A|(1+1/\ell)^{k\ell} 
\le e^k |A|  \le (2e\ell)^k |B|. \]
Considering $\frac{2|\ell\cdot B|}{\ell^{k} |B|}
=\prod_{i=0}^{j-1} \frac{|2^{i+1}\cdot B|}{2^k|2^i \cdot B|}$,  
we see that we can choose $\ell' = 2^{i'}$ such that
$\frac{|2\ell' \cdot B|}{2^k|\ell' \cdot B|} \le (2(2e)^k)^{1/j} < 1 + 2^{-2k}$, 
as $j \ge \log_2 L \gg k$. By Freiman's Theorem,
we deduce $\ell' \cdot B \sub P'$ for some $80$-proper $d$-GAP $P'$ 
with $d \le k$ and $|P'| \le O_k(|\ell' \cdot B|) \le O_k(\ell'^k |B|)$,
where the final inequality holds as
otherwise we would contradict $ |\ell\cdot B| \le O_k(\ell^k |B|)$
by repeated application of \Cref{iteratedBilu} (Freiman-Bilu for iterated sumsets).

By the non-degeneracy condition
$|\gap_{\ell' n}^{k-1}(\ell' \cdot B)|\geq n\ell'^k|B| \gg |P|$
we have $d=k$ and $P'$ is $\ell' n$-full.
As $P'$ is $80$-proper we have
 $\ell' \cdot B \sub P' \sub \ell' \cdot P''$
 for some $40\ell'$-proper $n$-full $k$-GAP $P'' \sub P'$
 with $\ell'^k |P''| = O_k(|P'|)$. 
 Applying \Cref{shrinkingthebox} we find
 $B \sub P$ for some translate $P$ of $20 \cdot P''$,
where $|P|=O_k(|P''|) = O_k(|B|)$.
 
We let $Q = \ell \cdot P$, noting that $\ell \cdot B \sub Q$ and $|Q| \le O_k(|\ell \cdot B|)$.
Recalling  $ |A'+\ell\cdot B| = O_k(|\ell \cdot B|)$, by the Ruzsa Covering Lemma
we have $A' \sub X + Q-Q$ with $|X| = O_k(1)$.
Next we apply a merging process, where we start with $X'=X$ and $Q'=Q-Q$,
and if we find any $x,y \in X'$ such that $x+Q'+P$ intersects $y+Q'+P$
then we replace $X'$ by $X' \sm \{y\}$ and $Q'$ by $6 \cdot Q'$,
noting that $y+Q'+P \sub (x+Q'-Q'+P-P) + Q'+P \sub x + 6 \cdot Q'$.
Thus we terminate with $A' \sub X' + Q'$ for some $X' \sub X$ 
with $|Q'| \le O_k(|Q|)  \le O_k(|\ell \cdot B|)$.

By the non-degeneracy condition
$|\gap_{\ell n}^{k-1}(\ell \cdot B)|\geq n\ell^k|B| \gg |Q'|$
we can assume that $Q'$ is $2$-proper
(otherwise $Q'$ would be contained in a $(k-1)$-GAP
of size $O_k(|Q'|)$ - see e.g.~\cite[Lemma 5.4]{van2023locality}).
Recalling that $P$ is $n$-full, applying \Cref{kconcaveprekopa} 
as in the proof of \Cref{stab1} we deduce 
\[ |A'+B|^{1/k} \ge |A'|^{1/k} + (|B|-kn^{-1}|P|)^{1/k}.\]
As $|P|=O_k(|P''|) = O_k(|B|)$ this concludes the proof.
\end{proof}

\section{Concluding remarks}

Given \Cref{quantcor} it would be interesting to determine the optimal constants $c^{\ref{quantcor}}_{k,t}$. In this direction the following asymptotically optimal result was derived in \cite{van2023locality}.
\begin{thm}
For all $k\in\mathbb{N}$, there exists $e_{k}$, and for all $n\in\mathbb{N}$ there exist $m_{n,k}$ so that if $A\subset\mathbb{Z}$ satisfies 
$|\gap_{m_{n,k}}^{k-2}(A)|,|\gap_{n}^{k-1}(A)|\geq e_{k}|A|$ then $|A+A|\geq 2^k(1-(1+o(1))\frac{k}{4}n^{-1})|A|$, where $o(1)\to 0$ as $n\to \infty$.
\end{thm}
This result establishes, up to the $1+o(1)$ term, the optimal constant as shown by considering a discrete cone (see \cite[Example 3.19]{van2023locality}).
However, in the context of this paper it is natural to ask for the optimal constant without
the additional condition $|\gap_{m_{n,k}}^{k-2}(A)|\geq e_{k}|A|$, the discrete simplex $S_n:=\{x\in[0,n]^k: \sum_i x_i \leq n\}$ gives a worse constant. We leave you with the following question.
\begin{quest}
What is the smallest constant $c_{k,t}$ for which \Cref{quantcor} holds?
In particular, among all sets in $\mathbb{Z}^k$ not contained in $n$ hyperplanes does the discrete simplex $S_n$ have minimal doubling?
\end{quest}

\bibliographystyle{alpha}
\bibliography{references}

\end{document}